\documentclass[12pt]{amsart}
\usepackage{amsfonts,amssymb,amscd,amsmath,enumerate,verbatim}

\def\NZQ{\mathbb}               

\def\ZZ{{\NZQ Z}}
\def\RR{{\NZQ R}}

\def\frk{\mathfrak}               

\def\Phi{{\frk N}}
\def\eb{{\bold e}}

\def\opn#1#2{\def#1{\operatorname{#2}}} 
\opn\chara{char} 
\opn\length{\ell} 
\opn\pd{pd} 
\opn\rk{rk}
\opn\projdim{proj\,dim} 
\opn\injdim{inj\,dim} 
\opn\rank{rank}
\opn\depth{depth} 
\opn\grade{grade} 
\opn\height{height}
\opn\embdim{emb\,dim} 
\opn\codim{codim}

\opn\Tr{Tr} 
\opn\bigrank{big\,rank}
\opn\superheight{superheight}
\opn\lcm{lcm}
\opn\trdeg{tr\,deg}
\opn\reg{reg} 
\opn\lreg{lreg} 
\opn\ini{in} 
\opn\lpd{lpd}
\opn\size{size}
\opn\mult{mult}
\opn\dist{dist}
\opn\cone{cone}
\opn\lex{lex}
\opn\rev{rev}
\opn\div{div} \opn\Div{Div} \opn\cl{cl} \opn\Cl{Cl}
\opn\Spec{Spec} \opn\Supp{Supp} \opn\supp{supp} \opn\Sing{Sing}
\opn\Ass{Ass} \opn\Min{Min}
\opn\Ann{Ann} \opn\Rad{Rad} \opn\Soc{Soc}
\opn\Syz{Syz} \opn\Im{Im} \opn\Ker{Ker} \opn\Coker{Coker}
\opn\Am{Am} \opn\Hom{Hom} \opn\Tor{Tor} \opn\Ext{Ext}
\opn\End{End} \opn\Aut{Aut} \opn\id{id} \opn\ini{in}

\opn\nat{nat}
\opn\pff{pf}
\opn\Pf{Pf} \opn\GL{GL} \opn\SL{SL} \opn\mod{mod} \opn\ord{ord}
\opn\Gin{Gin}
\opn\Hilb{Hilb}\opn\adeg{adeg}\opn\std{std}\opn\ip{infpt}
\opn\Pol{Pol}
\opn\sat{sat}
\opn\Var{Var}
\opn\Gen{Gen}

\opn\aff{aff} \opn\con{conv} \opn\relint{relint} \opn\st{st}
\opn\lk{lk} \opn\cn{cn} \opn\core{core} \opn\vol{vol}
\opn\link{link} \opn\star{star}
\opn\gr{gr}

\def\pot#1#2{#1[\kern-0.28ex[#2]\kern-0.28ex]}

\opn\dirlim{\underrightarrow{\lim}}
\opn\inivlim{\underleftarrow{\lim}}
\def\Implies{\ifmmode\Longrightarrow \else
	\unskip${}\Longrightarrow{}$\ignorespaces\fi}
\def\implies{\ifmmode\Rightarrow \else
	\unskip${}\Rightarrow{}$\ignorespaces\fi}
\def\iff{\ifmmode\Longleftrightarrow \else
	\unskip${}\Longleftrightarrow{}$\ignorespaces\fi}

\let\:=\colon
\newtheorem{Theorem}{Theorem}[section]
\newtheorem{Lemma}[Theorem]{Lemma}
\newtheorem{Corollary}[Theorem]{Corollary}
\newtheorem{Proposition}[Theorem]{Proposition}

\theoremstyle{remark}
\newtheorem{Remark}[Theorem]{Remark}
\let\epsilon\varepsilon
\let\phi=\varphi
\let\kappa=\varkappa
\def\qed{\ifhmode\textqed\fi
	\ifmmode\ifinner\quad\qedsymbol\else\dispqed\fi\fi}
\def\textqed{\unskip\nobreak\penalty50
	\hskip2em\hbox{}\nobreak\hfil\qedsymbol
	\parfillskip=0pt \finalhyphendemerits=0}
\def\dispqed{\rlap{\qquad\qedsymbol}}

\opn\dis{dis}
\opn\height{height}
\opn\dist{dist}
\def\pnt{{\raise0.5mm\hbox{\large\bf.}}}

\opn\Lex{Lex}

\begin{document}
\title{Existence of regular unimodular triangulations of dilated empty simplices}
\author{Takayuki Hibi, Akihiro Higashitani and Koutarou Yoshida}
\thanks{
{\bf 2010 Mathematics Subject Classification:}
Primary 52B20; Secondary 52B11. \\
\, \, \, {\bf Keywords:}
regular unimodular triangulation, empty simplex, integer decomposition property, toric ideal, Gr\"obner basis \\
The second author is partially supported by JSPS Grant-in-Aid for Young Scientists (B) $\sharp$17K14177. }

\address{Takayuki Hibi,
Department of Pure and Applied Mathematics,
Graduate School of Information Science and Technology, Osaka University,
Suita, Osaka 565-0871, Japan}
\email{hibi@math.sci.osaka-u.ac.jp}
\address{Akihiro Higashitani,
Department of Mathematics, Kyoto Sangyo University, 
Motoyama, Kamigamo, Kita-Ku, Kyoto, Japan, 603-8555}
\email{ahigashi@cc.kyoto-su.ac.jp}
\address{Koutarou Yoshida, 
Department of Pure and Applied Mathematics, Graduate School of Information Science and Technology, 
Osaka University, Suita, Osaka 565-0871, Japan} 
\email{yoshidakiiciva08@yahoo.co.jp}

\maketitle

\begin{abstract}
Given integers $k$ and $m$ with $k \geq 2$ and $m \geq 2$, 
let $P$ be an empty simplex of dimension $(2k-1)$ whose $\delta$-polynomial is of the form $1+(m-1)t^k$. 
In the present paper, the necessary and sufficient condition for the $k$-th dilation $kP$ of $P$ 
to have a regular unimodular triangulation will be presented. 
\end{abstract}

\section{Introduction}

\subsection{Integral convex polytopes and $\delta$-polynomials} 
An {\em integral} convex polytope is a convex polytope whose vertices are integer points. 
For an integral convex polytope $P \subset \RR^d$ of dimension $d$, 
we consider the generating function $\sum_{n \geq 0}|n P \cap \ZZ^d|t^n$, where $nP=\{n \alpha \;|\; \alpha \in P\}$. 
Then it is well-known that this becomes a rational function of the form 
$$\sum_{n \geq 0}|n P \cap \ZZ^d|t^n=\frac{\delta_P(t)}{(1-t)^{d+1}},$$
where $\delta_P(t)$ is a polynomial in $t$ of degree at most $d$ with nonnegative integer coefficients. 
The polynomial $\delta_P(t)$ is called the {\em $\delta$-polynomial}, also known as the {\em (Ehrhart) $h^*$-polynomial} of $P$. 
For more details on $\delta$-polynomials of integral convex polytopes, please refer to \cite{BeckRobins} or \cite{HibiBook}.

\subsection{Empty simplices} 
An integral simplex $P \subset \RR^d$ is called {\em empty} if $P$ contains no integer point except for its vertices. 
Note that $P$ is an empty simplex if and only if the linear term of $\delta_P(t)$ vanishes. 
Empty simplices are of particular interest in the area of not only combinatorics on integral convex polytopes but also toric geometry. 
Especially, the characterization problem of empty simplices is one of the most important topics. 
Originally, the empty simplices of dimension $3$ were completely characterized by G. K. White (\cite{White}). 
Note that the $\delta$-polynomial of every empty simplex of dimension $3$ is of the form $1+(m-1)t^2$ for some positive integer $m$. 
Recently, this characterization of empty simplices has been generalized by Batyrev--Hofscheier \cite{BH1}. 
More precisely, the following theorem has been proved. 
\begin{Theorem}[{cf. \cite[Theorem 2.5]{BH1}}]\label{empty}
Given an integer $k \geq 2$, let $d=2k-1$. 
Then $P \subset \RR^d$ is an empty simplex of dimension $d$ whose $\delta$-polynomial is of the form $1+(m-1)t^k$ for some $m \geq 2$ 
if and only if there are integers $a_1,\ldots,a_{k-1}$ with $1 \leq a_i \leq m/2$ and $(a_i,m)=1$ for each $1 \leq i \leq k-1$ 
such that $P$ is unimodularly equivalent to the convex hull of 
\begin{align}\label{abm}
\left\{{\bf 0},\eb_1,\ldots,\eb_{d-1},\sum_{i=1}^{k-1}a_i\eb_i+\sum_{j=k}^{d-1}(m-a_{d-j})\eb_j+m\eb_d\right\}.
\end{align}
\end{Theorem}
Here $(a,b)$ denotes the greatest common divisor of two positive integers $a$ and $b$, 
$\eb_1,\ldots,\eb_d \in \RR^d$ are the unit coordinate vectors of $\RR^d$ and ${\bf 0} \in \RR^d$ is the origin. 

Given integers $a_1,\ldots,a_{k-1},m$ with $1 \leq a_i \leq m/2$ and $(a_i,m)=1$ for each $i$, 
let $P(a_1,\ldots,a_{k-1},m)$ denote the convex hull of \eqref{abm}.

\subsection{The integer decomposition property and unimodular triangulations} 
We say that an integral convex polytope $P \subset \RR^d$ has the {\em integer decomposition property} (IDP, for short) 
if for each integer $n \geq 1$ and for each $\gamma \in nP \cap \ZZ^d$, 
there exist $\gamma^{(1)},\ldots,\gamma^{(n)}$ belonging to $P \cap \ZZ^d$ such that $\gamma=\gamma^{(1)}+\cdots+\gamma^{(n)}$. 

Under the assumption that the affine lattice generated by $P \cap \ZZ^d$ is equal to the whole lattice $\ZZ^d$, 
the following implications for integral convex polytopes hold: 
\begin{center}
a regular unimodular triangulation $\Rightarrow$ a unimodular triangulation 

$\Rightarrow$ a unimodular covering $\Rightarrow$ IDP 
\end{center}
(Please refer the reader to \cite{Stur} for the notions of (regular) unimodular triangulation or unimodular covering.) 
Note that for each implication, there exists an example of an integral convex polytope not satisfying the converse 
(see \cite{OH99}, \cite{FZ99} and \cite{BG99}).

\subsection{Motivation and results} 
For any integral convex polytope $P$ of dimension $d$, we know by \cite[Theorem 1.3.3]{BGT} that 
$nP$ always has IDP for every $n \geq d-1$. 
Moreover, we also know by \cite[Theorem 1.3.1]{BGT} that there exists a constant $n_0$ such that 
$nP$ has a unimodular covering for every $n \geq n_0$. 
However, it is still open whether there really exists a constant $n_0$ such that 
$nP$ has a (regular) unimodular triangulation for every $n \geq n_0$, 
while it is only known that there exists a constant $c$ such that $cP$ has a unimodular triangulation (\cite[Theorem 4.1 (p. 161)]{KKMS}). 

On the other hand, it is proved in \cite{KS} and \cite{SZ14} that for any $3$-dimensional integral convex polytope $P$, 
$nP$ has a unimodular triangulation for $n=4$ (\cite{KS}) and every $n \geq 6$ (\cite[Theorem 1.4]{SZ14}). 
For the proofs of those results, the discussions about the existence of a (regular) unimodular triangulation 
of the {\em dilated} empty simplices of dimension $3$ are crucial, where 
for an empty simplex $P$, a dilated empty simplex means a simplex $nP$ for some positive integer $n$. 
Hence, for the further investigation of higher-dimensional cases, the existence of a unimodular triangulation of 
dilated empty simplices might be important. 
Since $P(a_1,\ldots,a_{k-1},m)$ can be understood as a generalization of empty simplices of dimension $3$, 
it is quite reasonable to study the existence of a unimodular triangulation of the dilated empty simplex. 
Moreover, since $kP(a_1,\ldots,a_{k-1},m)$ has IDP but $k'P(a_1,\ldots,a_{k-1},m)$ does not for any $k'<k$ (see Proposition \ref{prop:IDP}), 
it is natural to discuss the existence of a unimodular triangulation of $kP(a_1,\ldots,a_{k-1},m)$. 

The purpose of the present paper is to show the following: 
\begin{Theorem}[Main Theorem]\label{mainthm}
Given an integer $k \geq 2$, let $d=2k-1$ and let $P \subset \RR^d$ be an empty simplex 
whose $\delta$-polynomial is of the form $1+(m-1)t^k$ for some $m \geq 2$. 
Then $kP$ has a regular unimodular triangulation if and only if $P$ is unimodularly equivalent to the convex hull of 
$$\left\{{\bf 0},\eb_1,\ldots,\eb_{d-1},\sum_{i=1}^{k-1}\eb_i+(m-1)\sum_{j=k}^{d-1}\eb_j+m\eb_d\right\}.$$ 
\end{Theorem}
\begin{Remark}
This theorem can be regarded as a generalization of (a part of) \cite[Corollary 3.5]{SZ14}. 
\end{Remark}

Theorem \ref{mainthm} says that $P$ is an empty simplex with $\delta_P(t)=1+(m-1)t^k$ for some $m \geq 2$ 
whose $k$-th dilation has a regular unimodular triangulation if and only if $P$ is unimodularly equivalent to $P(1,\ldots,1,m)$. 
The proof of the necessity of Theorem \ref{mainthm} is given in Section \ref{nec} (Proposition \ref{prop:nec}) and 
the sufficiency is given in Section \ref{suf} (Proposition \ref{prop:suf}), respectively. 

Furthermore, by Proposition \ref{prop:IDP} together with Theorem \ref{mainthm}, we obtain 
\begin{Corollary}
For any integer $k \geq 2$, there exists an empty simplex $P$ of dimension $(2k-1)$ 
such that $kP$ has IDP but has no regular unimodular triangulation. 
\end{Corollary}

\bigskip

\section{Integer decomposition property of $P(a_1,\ldots,a_{k-1},m)$} 

Before proving the main theorem (Theorem \ref{mainthm}), we prove the following: 
\begin{Proposition}\label{prop:IDP}
Let $P=P(a_1,\ldots,a_{k-1},m)$, where $a_1,\ldots,a_{k-1}$, $m$ are integers with $1 \leq a_i \leq m/2$ and $(a_i,m)=1$ for each $i$. 
Then $nP$ has IDP for a positive integer $n$ if and only if $n \geq k$. 
\end{Proposition}
\begin{proof}
Let $v_0={\bf 0}$, $v_j=\eb_j$ for $1 \leq j \leq d-1$ and $v_d=\sum_{i=1}^{k-1}a_i\eb_i+\sum_{j=k}^{d-1}(m-a_{d-j})\eb_j+m\eb_d$. 
We define $w_i$ by setting 
\begin{align*}
\left(\cfrac{\overline{i(m-a_1)}+ia_1}{m},\ldots,\cfrac{\overline{i(m-a_{k-1})}+ia_{k-1}}{m},
\cfrac{\overline{ia_{k-1}}+i(m-a_{k-1})}{m},\ldots,\cfrac{\overline{ia_1}+i(m-a_1)}{m},i\right)\end{align*}
for each $1 \leq i \leq m-1$, where $\overline{\ell}$ denotes the remainder of $\ell$ divided by $m$. 
Then we can see that 
\begin{align*}
w_i=\frac{m-i}{m}v_0+\sum_{p=1}^{k-1}\frac{\overline{i(m-a_p)}}{m}v_p+\sum_{q=k}^{d-1}\frac{\overline{ia_{d-q}}}{m}v_q+\frac{i}{m}v_d
\end{align*}
and $\displaystyle \sum_{j=1}^{k-1}\left(\frac{\overline{i(m-a_j)}}{m}+\frac{\overline{ia_j}}{m}\right)+\frac{m-i}{m}+\frac{i}{m}=k$, 
which says that $w_i \in kP \cap \ZZ^d$ for each $i$. 

Now, we see that 
\begin{align}\label{kp}
nP \cap \ZZ^d=(\underbrace{P \cap \ZZ^d + \cdots +P \cap \ZZ^d}_n) \sqcup 
(\underbrace{P \cap \ZZ^d + \cdots +P \cap \ZZ^d}_{n-k} + \{w_i\}_{1 \leq i \leq m-1}) \end{align}
for every $n \geq k$. In fact, for any $\alpha \in nP \cap \ZZ^d$, we can write $\alpha=\sum_{i=0}^dr_iv_i$, 
where $r_i \geq 0$ for each $i$ and $\sum_{i=0}^d r_i = n$. 
Let $\beta=\sum_{i=0}^d \lfloor r_i \rfloor v_i$ and $\beta'=\sum_{i=0}^d (r_i-\lfloor r_i \rfloor) v_i$. 
Note that $0 \leq r_i - \lfloor r_i \rfloor <1$. 
When $\beta'={\bf 0}$, one has $\alpha=\beta \in \underbrace{P \cap \ZZ^d + \cdots + P \cap \ZZ^d}_n$. 
Assume $\beta' \neq {\bf 0}$. Now, it is well-known that 
the $\delta$-polynomial of an integral simplex can be computed as follows: 
for an integral simplex $P \subset \RR^d$ of dimension $d$ with its $\delta$-polynomial $\sum_{i=0}^d \delta_it^i$, 
we have $\delta_j=|\{ \alpha \in \ZZ^d \;|\; \alpha = \sum_{i=0}^d s_iv_i, 0 \leq s_i < 1, \sum_{i=0}^d s_i=j\}|$ for each $0 \leq j \leq d$. 
(Consult, e.g., \cite[Proposition 27.7]{HibiBook}.) 
In our case, since $\delta_P(t)=1+(m-1)t^k$, we see the equality 
\begin{align}\label{equality}
\left\{\alpha \in \ZZ^d \;\bigg\vert\; \alpha = \sum_{i=0}^d s_iv_i, 0 \leq s_i < 1, \sum_{i=0}^d s_i \in \ZZ_{>0}\right\}=\{w_1,\ldots,w_{m-1}\}.
\end{align}
Hence, we obtain $\beta'$ belongs to $\{w_1,\ldots,w_{m-1}\}$ and 
$$\alpha=\beta+\beta' \in \underbrace{P \cap \ZZ^d + \cdots +P \cap \ZZ^d}_{n-k} + \{w_i\}_{1 \leq i \leq m-1}.$$ 

By \eqref{kp}, we can discuss as follows: 
\begin{itemize}
\item When $n \geq k$, let $\gamma \in \ell(nP) \cap \ZZ^d$ for $\ell \geq 1$. Since $\ell n \geq n \geq k$, 
it follows from \eqref{kp} that there exist $\gamma^{(1)},\ldots,\gamma^{(\ell)}$ belonging to $nP \cap \ZZ^d$ such that 
$\gamma=\gamma^{(1)}+\cdots+\gamma^{(\ell)}$. This means that $nP$ has IDP. 
\item When $n<k$, let $\ell'$ be a minimum positive integer with $\ell' n \geq k$. 
Then $w_i+(\ell'n-k)v_0 \in \ell' n P \cap \ZZ^d$ for each $i$, 
while $w_i+(\ell'n-k)v_0$ cannot be written as a sum of $\ell'$ elements in $n P \cap \ZZ^d$ because we have 
$w_i \not\in (k'P \cap \ZZ^d) + ((k-k')P \cap \ZZ^d)$ for any $1 \leq k' \leq k-1$, which follows from \eqref{equality}, 
and $w_i \in kP \cap \ZZ^d$. 
This means that $nP$ does not have IDP. 
\end{itemize}
Therefore, we conclude that $nP$ has IDP if and only if $n \geq k$. 
\end{proof}

From \eqref{kp}, we obtain the following which we will use later: 
\begin{align*}
kP\cap \ZZ^d=(\underbrace{P\cap\ZZ^d+\cdots+P\cap\ZZ^d}_{k})\sqcup\left\{w_i\right\}_{1 \leq i \leq m-1}.
\end{align*}
Moreover, from this equation we can also see that the affine lattice generated by $kP \cap \ZZ^d$ becomes $\ZZ^d$. 
In fact, since ${\bf 0}=kv_0 \in kP \cap \ZZ^d$, ${\bf e}_i=(k-1)v_0+v_i \in kP \cap \ZZ^d$ for each $1 \leq i \leq d-1$ and 
$w_1=\sum_{i=1}^d \eb_i \in kP \cap \ZZ^d$, we observe that the lattice points $\eb_1,\ldots,\eb_{d-1}, w_1$ generate $\ZZ^d$.

\begin{Remark}
In \cite{CHHH}, several invariants concerning the dilation of integral convex polytopes are studied. 
For the case $P=P(a_1,\ldots,a_{k-1},m)$, we can show that $\mu_\text{va}(P) =\mu_\text{Ehr}(P)= k$. 
Hence, by \cite[Theorem 1.1]{CHHH}, we obtain that all invariants defined there are equal to $k$. 
\end{Remark}


\bigskip
\section{Proof of Theorem \ref{mainthm} : the necessity}\label{nec}

This section is devoted to giving a proof of the necessity of Theorem \ref{mainthm}. We prove the following: 
\begin{Proposition}\label{prop:nec}Given an integer $k \geq 2$, let $d=2k-1$, let $m \geq 2$ be an integer and 
let $a_i,b_i$ be positive integers with $a_i+b_i=m$ and $(a_i,m)=1$ for each $1 \leq i \leq k$. 
Let $$P=\con\left(\left\{\mathbf{0},\mathbf{e}_1,\ldots,\mathbf{e}_{d-1},
\sum_{i=1}^{k-1}a_i\eb_i+\sum_{j=k}^{d-1}b_{d-j}\eb_j+m\eb_d\right\}\right) \subset \RR^d.$$ 
Assume that there is $j$ such that $2 \leq a_j \leq m-2$. Then $kP$ does not have a regular unimodular triangulation.
\end{Proposition}

Similar to the previous section, 
let $v_0={\bf 0}$, $v_j=\eb_j$ for $1 \leq j \leq d-1$, $v_d=\sum_{i=1}^{k-1}a_i\eb_i+\sum_{j=k}^{d-1}b_{d-j}\eb_j+m\eb_d$ 
and let
                \[w_i=\left(\cfrac{\overline{ib_1}+ia_1}{m},\ldots,\cfrac{\overline{ib_{k-1}}+ia_{k-1}}{m},
                \cfrac{\overline{ia_{k-1}}+ib_{k-1}}{m},\ldots,\cfrac{\overline{ia_1}+ib_1}{m},i\right) \in \RR^d\] 
for each $1 \leq i \leq m-1$. Let $A=\underbrace{P \cap \ZZ^d+\cdots+P \cap \ZZ^d}_k$. 
Then $$kP \cap \ZZ^d=A \sqcup \{w_i\}_{1 \leq i \leq k}.$$ 

For the proof of Proposition \ref{prop:nec}, we prepare three lemmas (Lemma \ref{lem:yoshida1}, Lemma \ref{lem:yoshida2} and Lemma \ref{lem:yoshida3}). 
In the proofs of those lemmas, we will use the following notation. 
For $x \in \RR^d$, let $p_j(x)$ denote the $j$-$th$ coordinate of $x$. 
Remark that for $1 \leq i \leq m-1$ and $1 \leq j \leq k-1$, we have $p_j(w_i)+p_{2k-1-j}(w_i)=i+1$                   
\begin{Lemma}\label{lem:yoshida1}
For any $2 \leq i \leq m-2$, there exist $v,v^\prime$ $\in A$ which satisfy the equalities 
$w_{i-1}+w_{i+1}-2w_i=v-v^\prime$ and $w_{m-(i-1)}+w_{m-(i+1)}-2w_{m-i}=v^\prime-v$. 
\end{Lemma}
\begin{proof}
For simplicity, we denote $w^*= w_{i-1}+w_{i+1}-2w_i$. If  $w^*=\mathbf{0}$, then the assertion is obvious, so we suppose $w^*\neq\mathbf{0}$. 

It follows from an easy calculation that                 
\begin{align*}
p_j(w^*)=\begin{cases}
\cfrac{\overline{(i-1)b_j}+\overline{(i+1)b_j}-2\overline{ib_j}}{m}, \; &\textnormal{if } 1 \leq j \leq k-1, \\
\cfrac{\overline{(i-1)a_{2k-j-1}}+\overline{(i+1)a_{2k-j-1}}-2\overline{ia_{2k-j-1}}}{m},  &\textnormal{if } k \leq j \leq 2k-2, \\
0, &\textnormal{if } j = 2k-1. 
\end{cases}\end{align*} 
Hence we can see that $-1 \leq  p_j(w^*) \leq 1$ and $p_j(w^*)+p_{2k-1-j}(w^*)=0$ for $1 \leq j \leq 2k-2$. 
Without loss of generality, we may assume $p_1(w^*)=1$ by $w^*\neq\mathbf{0}$. We define $v,v^\prime \in \RR^d$ as follows: 
\begin{align*}
p_j(v)&=\begin{cases}
2, \; &\text{if }  j=1, \\
1, &\text{if }    j\geq 2  \text{ and } p_j(w^*)=1,  \\
0, &\text{if }    j\geq 2  \text{ and }  p_j(w^*)=-1,0, \\
\end{cases}
\end{align*} and \begin{align*}
p_j(v^\prime)&=\begin{cases}
1, \; &\textnormal{if } j=1, \\
1, &\text{if } j\geq 2 \text{ and } p_j(w^*)=-1, \\
0, &\textnormal{if } j\geq 2 \text{ and } p_j(w^*)=1,0.
\end{cases}
\end{align*}
Then we can verify that $w_{i-1}+w_{i+1}-2w_i=v-v^\prime$. Since $p_j(w^*)+p_{2k-1-j}(w^*)=0$ for $1 \leq j \leq 2k-2$, 
one has $\sum\limits_{j=1}^{2k-1} p_j(v), \sum\limits_{j=1}^{2k-1} p_j(v^\prime) \leq k$. Moreover, 
$p_{2k-1}(v)=p_{2k-1}(v^\prime)=0$. Thus we know that $v$ and $v^\prime$ are contained in $A$. 
Since we see that $\overline{(m-i)b_j}+\overline{ib_j}=\overline{(m-i)a_j}+\overline{ia_j}=m$ 
for $1 \leq i \leq m-1$ and $1 \leq j \leq k-1$, we obtain $ p_j(w_{i-1}+w_{i+1}-2w_i)=- p_j(w_{m-(i-1)}+w_{m-(i+1)}-2w_{m-i})$.
Hence we have that $w_{m-(i-1)}+w_{m-(i+1)}-2w_{m-i}=v^\prime-v$. 
\end{proof}
\begin{Lemma}\label{lem:yoshida2}
Let $2 \leq a \leq m-2$ be an integer and suppose $\overline{aa^\prime}=1$, where $2 \leq a^\prime \leq m-2$. 
Then there exist $u,u^\prime \in A$ which satisfy
$w_{\overline{(a-1)a^\prime}}+w_{\overline{(a+1)a^\prime}}-2w_1=u-u^\prime$ and 
$w_{\overline{((m-1)a-1)a^\prime}}+w_{\overline{((m-1)a+1)a^\prime}}-2w_{m-1}=u^\prime-u$. 
\end{Lemma}
\begin{proof}
Let $w^{**}=w_{\overline{(a-1)a^\prime}}+w_{\overline{(a+1)a^\prime}}-2w_1$. 
At first, we show that 
\begin{align*}
p_j(w^{**})=\begin{cases}
a_j-1 \text{ or } a_j \text{ or }a_j+1, &\text{if }1\leq j \leq k-1, \\
b_{2k-1-j}-1 \text{ or }b_{2k-1-j} \text{ or }b_{2k-1-j}+1, &\text{if } k\leq j \leq 2k-2, \\
m, &\textnormal{if } j=2k-1. 
\end{cases}\end{align*} 

We see that $p_{2k-1}(w^{**})=\overline{(a-1)a^\prime}+\overline{(a+1)a^\prime}-2=\overline{1-a'}+\overline{1+a^\prime}-2=m+1-a^\prime+1+a^\prime-2=m$.
Since $p_i(w^{**})+p_{2k-1-i}(w^{**})=\overline{(a-1)a^\prime}+1+\overline{(a+1)a^\prime}+1-4=m$, we consider $p_i(w^{**})$ only for $1 \leq i \leq k-1$.
For  $1 \leq i \leq k-1$, 
\begin{align*}
p_i(w^{**})&=\cfrac{\overline{\overline{(a-1)a^\prime}b_i}+\overline{(a-1)a^\prime}a_i+\overline{\overline{(a+1)a^\prime}b_i}+\overline{(a+1)a^\prime}a_i}{m}-2\\
&=\cfrac{\overline{b_i-a^\prime b_i}+\overline{b_i+a^\prime b_i}+(m+2)a_i}{m}-2. 
\end{align*}
Since $\overline{b_i-a^\prime b_i}+\overline{b_i+a^\prime b_i}=2b_i+m$ or $2b_i$ or $2b_i-m$, 
one sees that $p_i(w^{**})=a_i+1$ or $a_i$ or $a_i-1$. Note that for $i$ with $a_i=a$, we see that $p_i(w^{**})=a_i$. 

When $w^{**}=v_d$, we may set $u=v_d$ and $ u^\prime=\mathbf{0}$. Hence, without loss of generality, we assume $p_1(w^{**})=a_1+1$. 
In addition, since $2 \leq a_i=a \leq m-2$ for some $i$, we may also assume $i=2$, i.e., $a_2=a$. Then we have $p_2(w^{**})=a_2=a$. 
We define $u,u^\prime \in \RR^d$ as follows: 
\begin{align*}
p_j(u'')=\begin{cases}
2, \; &\textnormal{if }j=1, \\ 
0, \; &\textnormal{if } j=2, \\
1, \; &\textnormal{if } 3 \leq j \leq 2k-2 \text{ and }p_j(w^{**}-v_d)=1, \\
0, \; &\textnormal{if } 3 \leq j \leq 2k-2 \text{ and }p_j(w^{**}-v_d)=-1\text{ or }0, \\
0, \; &\textnormal{if } j=2k-1, 
\end{cases}
\end{align*} 
and 
\begin{align*}
p_j(u^\prime)=\begin{cases}
1, \; &\textnormal{if }j=1, \\ 
0, \; &\textnormal{if } j=2, \\
1, \; &\textnormal{if } 3 \leq j \leq 2k-2 \text{ and }p_j(w^{**}-v_d)=-1, \\
0, \; &\textnormal{if } 3 \leq j \leq 2k-2 \text{ and }p_j(w^{**}-v_d)=1\text{ or }0, \\
0, \; &\textnormal{if } j=2k-1 
\end{cases}
\end{align*} 
and take $u=v_d+u''$. From the above discussion and the equalities 
$\overline{(a-1)a^\prime}+\overline{(\overline{(m-1)a}+1)a^\prime}=\overline{(a+1)a^\prime}+\overline{(\overline{(m-1)a}-1)a^\prime}=m$,
we see that $u$ and $u^\prime$ are the desired ones by the similar discussion in the proof of Lemma \ref{lem:yoshida1}.  
\end{proof}
\begin{Lemma}\label{lem:yoshida3}
For $v,v' \in A$ given in Lemma \ref{lem:yoshida1} and $u,u' \in A$ given in Lemma \ref{lem:yoshida2}, 
we consider $h_1,h_2 \in A$. 
\begin{itemize}
\item If $v+w_i=h_1+h_2$ for some $i$, then $v=h_1$ or $v=h_2$. 
\item If $v^\prime+w_i=h_1+h_2$ for some $i$, then $v^\prime=h_1$ or $v^\prime=h_2$. 
\item If $u+w_1=h_1+h_2$, then $u=h_1$ or $u=h_2$. 
\item If $u^\prime+w_{m-1}=h_1+h_2$, then $u^\prime=h_1$ or $u^\prime =h_2$. 
\end{itemize}
\end{Lemma}
\begin{proof}
We prove the first statement. The other statements are proved in the similar way. 
We assume that there exist $h_1$ and $h_2$ in $A$ satisfying $v+w_i=h_1+h_2$ for some $i$ and $v \neq h_1$ and $v \neq h_2$. 
Since $p_{2k-1}(v+w_i)=p_{2k-1}(h_1+h_2)=i$ and $2 \leq i \leq m-2$, we have $h_1=w_{i_1}$ and $h_2=w_{i_2}$, 
where $1 \leq i_1 \leq i_2 \leq m-1$. Therefore, 
$p_1(h_1+h_2)+p_{2k-2}(h_1+h_2)=p_1(h_1)+p_{2k-2}(h_1)+p_1(h_2)+p_{2k-2}(h_2)=p_{2k-1}(h_1)+1+p_{2k-1}(h_2)+1=p_{2k-1}(h_1+h_2)+2=i+2$, 
but $p_1(v+w_i)+p_{2k-2}(v+w_i)=p_1(v)+p_1(w_i)+p_{2k-2}(v)+p_{2k-2}(w_i)=i+3$, a contradiction. 
\end{proof}

Let $K$ be a field. 
Let $K[t_1^\pm, \ldots,t_d^\pm,s]$ denote the Laurent polynomial ring with $(d+1)$ variables. 
For an integral simplex $P \subset \RR^d$ in Proposition \ref{prop:nec}, 
if $\alpha =( \alpha_1,\ldots,\alpha_d) \in kP \cap \ZZ^d$, then we write $u_\alpha$ for the Laurent monomial 
$t_1^{\alpha_1}\cdots t_d^{\alpha_d} \in K[t_1^\pm, \ldots,t_d^\pm,s]$. The {\em Ehrhart ring} $K[kP]$ of $kP$ 
is the subring of $K[t_1,\ldots,t_d,s]$ generated by those monomials $u_\alpha s$ with $\alpha \in kP \cap \ZZ^d$. 
Note that this $K[kP]$ is usually called the {\em toric ring} of $kP$, but the toric ring of an integral convex polytope $Q$ 
coincides with the Ehrhart ring of $Q$ if and only if $Q$ has IDP, so we can call $K[kP]$ the Ehrhart ring. 
Let $S=K[ \left\{x_{i_1\cdots i_k}\right\}_{0 \leq i_1 \leq \cdots \leq i_k \leq d},\left\{y_j\right\}_{1 \leq j \leq m-1}]$ 
be the polynomial ring with $(\binom{d+k}{k}+m-1)$ variables with $\deg(x_{i_1\cdots i_k})=\deg(y_j)=1$. 
We define the surjective ring homomorphism 
$\pi : S \rightarrow K[kP]$ by setting $\pi(x_{i_1\cdots i_k})=u_{v_{i_1}+\cdots+v_{i_k}}s$ and $\pi(y_j)=u_{w_j}s$. 
Let $I$ denote the kernel of $\pi$ and we call $I$ the {\em toric ideal} of $kP$. 
It is known that $kP$ has a regular unimodular triangulation if and only if 
there exists a monomial order $<$ on $S$ such that the initial ideal $\ini_<(I)$ of $I$ with respect to $<$ is squarefree 
(e.g., see \cite[Corollary 8.9]{Stur}). 
In what follows, we will prove there is no such monomial order.

\begin{proof}[Proof of Proposition \ref{prop:nec}]
It follows from Lemma \ref{lem:yoshida1} that 
there exist variables $x_J$ and $x_{J'}$ of $S$ such that both $x_Jy_{i-1}y_{i+1}-x_{J'}y_i^2$ and $x_{J'}y_{m-(i-1)}y_{m-(i+1)}-x_J{y_{m-i}}^2$ 
belong to $I$ or both $y_{i-1}y_{i+1}-y_i^2$ and $y_{m-(i-1)}y_{m-(i+1)}-{y_{m-i}}^2$ belong to $I$ for each $2\leq i \leq m-2$. 
Moreover, it follows from Lemma \ref{lem:yoshida2} that there exist variables $x_L$ and $x_{L'}$ of $S$ such that 
both $x_Ly_{\overline{(a-1)a^\prime}}y_{\overline{(a+1)a^\prime}}-x_{L'}y_1^2$ and 
$x_{L'}y_{\overline{((m-1)a-1)a^\prime}}y_{\overline{((m-1)a+1)a^\prime}}-x_L{y_{m-1}}^2$ belong to $I$. 

On the contrary, suppose $kP$ has a regular unimodular triangulation, 
namely there exists a monomial order $<$ such that $\ini_<(I)$ is squarefree. 
Then, for all six binomials just appearing above, their initial monomials are the first monomials. 
In fact, for the four cubic binomials above, 
if the second monomial of one of those binomials is an initial monomial, 
since it is not squarefree but $\ini_<(I)$ is squarefree, 
the second monomial is divisible by a quadratic monomial belonging to $\ini_<(I)$. 
Hence, there exsits a binomial whose initial monomial is such quadratic monomial. However, this contradicts to Lemma \ref{lem:yoshida3}. 

Thus, we conclude that, for any monomial order $<$, one has 
$x_{J'}y_i^2 < x_Jy_{i-1}y_{i+1}$ and $x_J{y_{m-i}}^2 < x_{J'}y_{m-(i-1)}y_{m-(i+1)}$ for any $2\leq i \leq m-2$. 
Then $y_i^2{y_{m-i}}^2 < y_{i-1}y_{m-(i-1)}y_{i+1}y_{m-(i+1)}$ holds. 
Thus, we have $y_iy_{m-i} < y_{i-1}y_{m-(i-1)}$ or $ y_iy_{m-i} < y_{i+1}y_{m-(i+1)}$. Similarly, we also have 
$y_1{y_{m-1}} < y_{\overline{(a-1)a^\prime}}y_{\overline{((m-1)a+1)a^\prime}} $ or 
$y_1{y_{m-1}} < y_{\overline{(a+1)a^\prime}}y_{\overline{((m-1)a-1)a^\prime}}$,
and recall $\overline{(a-1)a^\prime}+\overline{((m-1)a+1)a^\prime}=m$
and $\overline{(a+1)a^\prime}+\overline{((m-1)a-1)a^\prime}=m$.
Thus, for any $1 \leq i \leq m-1$, there exists $1 \leq j \leq m-1$ such that $y_iy_{m-i} < y_jy_{m-j}$. This is a contradiction. 
\end{proof}

\bigskip


\section{Proof of Theorem \ref{mainthm} : the sufficiency}\label{suf}

This section is devoted to giving a proof of the sufficiency of Theorem \ref{mainthm}. We prove the following: 
\begin{Proposition}\label{prop:suf}
Given an integer $k \geq 2$, let $d=2k-1$. Let $m \geq 2$ be an integer and let 
\[P=\con\left(\left\{\mathbf{0},\mathbf{e}_1,\ldots,\mathbf{e}_{d-1}, 
\sum_{i=1}^{k-1}\eb_i+(m-1)\sum_{j=k}^{d-1}\eb_j+m\mathbf{e}_d\right\}\right) \subset \RR^d.\] 
Then $kP$ has a regular unimodular triangulation. 
\end{Proposition}
The strategy of our proof is to show the existence of a monomial order $<$ 
such that the toric ideal of $kP$ has a squarefree initial ideal. 
In what follows, we work with the same notation on the toric ideal of $kP$ as those in Section \ref{nec}.

\begin{proof}[Proof of Proposition \ref{prop:suf}]
Let $v_0=\mathbf{0}$, $v_i=\mathbf{e}_i$ for $1 \leq i \leq d-1$, $v_d=\sum_{i=1}^{k-1}\eb_i+(m-1)\sum_{j=k}^{d-1}\eb_j+m\mathbf{e}_d$ 
and let $w_j=\sum_{i=1}^{k-1}\eb_i+\ell\sum_{j=k}^d\eb_j$ for $1 \leq \ell \leq m-1$. 
Then we see that $$kP\cap \ZZ^d= \{v_{i_1}+\cdots+v_{i_k} \;|\; 0 \leq i_1 \leq \cdots \leq i_k \leq d \} 
\sqcup \left\{w_j \;|\; 1 \leq j \leq m-1 \right\}.$$

Let ${\bf u}_0=01 \cdots k-1$ and ${\bf u}_m=kk+1 \cdots d$ be the sequences of indices. 
Then we let $y_0=x_{{\bf u}_0}$ and $y_m=x_{{\bf u}_m}$ and we never use $x_{{\bf u}_0}$ and $x_{{\bf u}_m}$. 
Namely, for each variable $x_{{\bf s}}=x_{i_1 \cdots i_k}$ with $0 \leq i_1 \leq \cdots \leq i_k \leq d$ appearing below, 
we implicitly assume that ${\bf s} \neq {\bf u}_0$ and ${\bf s} \neq {\bf u}_m$. 
Moreover, we recall the notion of {\em sorting}. For a sequence $\ell_1,\ldots,\ell_p$, let 
$\mathrm{sort}(\ell_1\cdots\ell_p)$ denote the permutation $\ell_{i_1},\ldots,\ell_{i_p}$ of $\ell_1,\ldots,\ell_p$ with 
$\ell_{i_1} \leq \cdots \leq \ell_{i_p}$. 
We say that a monomial $x_{{\bf s}_1} \cdots x_{{\bf s}_\ell}$ is {\em sorted} 
if $\mathrm{sort}({\bf s}_1{\bf s}_2 \ldots {\bf s}_\ell)=
s_{1,1}s_{2,1}\cdots s_{\ell,1}s_{1,2} \cdots s_{\ell,2} \cdots s_{\ell,k}$, 
where ${\bf s}_i=s_{i,1}s_{i,2}\cdots s_{i,k}$ with $0 \leq s_{i,1} \leq \cdots \leq s_{i,k} \leq m$ for each $1 \leq i \leq \ell$. 

First, we define $(k+2)$ sets $G_{1,1},G_{1,2},G_{1,3}, G_2, \ldots, G_k$ of binomials as follows: 
\begin{align*}
G_{1,1}&=\left\{x_{{\bf s}_1}x_{{\bf s}_2}-x_{{\bf s}_1'}x_{{\bf s}_2'} \;|\; 
\mathrm{sort}({\bf s}_1{\bf s}_2)=\mathrm{sort}({\bf s}_1'{\bf s}_2') \right\}, \\ 
G_{1,2}&=\left\{x_{{\bf s}_1}x_{{\bf s}_2}-x_{{\bf s}_1'}y_p \;|\; p \in \{0,m\}, \;
\mathrm{sort}({\bf s}_1{\bf s}_2)=\mathrm{sort}({\bf s}_1' {\bf u}_p) \right\}, \\
G_{1,3}&=\{x_{{\bf s}_1}x_{{\bf s}_2}-y_py_q \;|\; p,q \in \{0,m\}, \; 
\mathrm{sort}({\bf s}_1{\bf s}_2)=\mathrm{sort}({\bf u}_p{\bf u}_q) \}, 
\end{align*}
where each $x_{\bf s}$ is of the form $x_{s_1 \cdots s_k}$ for some $0 \leq s_1 \leq \cdots \leq s_k \leq d$ 
and runs over all possible ${\bf s}$'s (but ${\bf s} \not\in\{{\bf u}_0,{\bf u}_m\}$), 
\begin{align*}
G_2&=\left\{y_p y_s -y_q y_r \;|\; 0\leq p \leq q \leq r\leq s \leq m , p+s=q+r \right\}, \;\text{and}\\
G_n&=\left\{ \begin{array}{l} x_{{\bf s}_1}x_{{\bf s}_2}\cdots x_{{\bf s}_n}-x_{{\bf t}_1}\cdots x_{{\bf t}_{n-1}}y_p \end{array} 
\left|
\begin{array}{l}
\bigcup_{\substack{1 \leq j \leq n \\ j \neq i}}\widetilde{{\bf s}_j}\not\supset \widetilde{{\bf u}_0}, \; 
\bigcup_{\substack{1 \leq j \leq n \\ j \neq i}}\widetilde{{\bf s}_j}\not\supset \widetilde{{\bf u}_m} \\
\text{for every } i=1,\cdots,n, \\
p \in \{0,m\}, \\
\mathrm{sort}({\bf t}_1{\bf t}_2\cdots {\bf t}_{n-1} {\bf u}_p) \\ =s_{1,1}s_{2,1}\cdots s_{n,1}s_{1,2} \cdots s_{n,2} \cdots s_{n,k}
\end{array}\right.\right\},
\end{align*}
for $3 \leq n \leq k$, where we denote $\widetilde{{\bf s}} = \left\{s_1,\ldots,s_k\right\}$ for ${\bf s}=s_1\cdots s_k$, 
and ${\bf s}_i=s_{i,1}s_{i,2}\cdots s_{i,k}$. 

Let $G=G_{1,1} \cup G_{1,2} \cup G_{1,3} \cup \bigcup_{i=2}^k G_i$.

Next, we define the monomial order on $S$ as follows: 
$\prod_{{\bf s}}x_{\bf s}\prod_ty_t < \prod_{{\bf s}'}x_{{\bf s}'}\prod_{t'}y_{t'} \iff$ 
\begin{enumerate}
\item[(i)] (the total degree of $\prod_{{\bf s}}x_{\bf s}$) $<$ (the total degree of $\prod_{{\bf s}'}x_{{\bf s}'}$), or 
\item[(ii)] (the total degree of $\prod_{{\bf s}}x_{\bf s}$) $=$ (the total degree of $\prod_{{\bf s}'}x_{{\bf s}'}$) 
and $\prod_{{\bf s}}x_{\bf s} < \prod_{{\bf s}'}x_{{\bf s}'}$ with respect to a sorting order (see \cite[Section 14]{Stur}), or 
\item[(iii)] $\prod_{\bf s}x_{\bf s} = \prod_{{\bf s}'}x_{{\bf s}'}$ and $\prod_ty_t < \prod_{t'}y_{t'}$ with respect to the lexicographic order 
induced by a ordering of variables $y_m<\cdots<y_0$. 
\end{enumerate}
We show that the initial monomial of each binomial in $G$ is squarefree. 
\begin{itemize}
\item On each binomial in $G_{1,1},G_{1,2},G_{1,3}$: 
by the property of a sorting order (the definition (ii) of the monomial order $<$), 
we know $\ini_<(x_{{\bf s}_1}x_{{\bf s}_2}-x_{{\bf s}_1'}x_{{\bf s}_2'})$ is squarefree. 
Regrading $x_{{\bf s}_1}x_{{\bf s}_2}-x_{{\bf s}_1'}y_p$, 
its initial monomial should be the first one by the definition (i) of $<$. 
If ${\bf s}_1={\bf s}_2$, i.e., the initial monomial is of the form $x_{\bf s}^2$ for some ${\bf s}=s_1\cdots s_k$, then
$\mathrm{sort}({\bf s}_1'{\bf u}_p)=s_1s_1 \cdots s_ks_k$, so ${\bf s}_1'$ should be also ${\bf u}_p$, a contradiction. 
Similarly, the initial monomial of $x_{{\bf s}_1}x_{{\bf s}_2}-y_py_q$ is also the first one and squarefree. 
\item On each binomial in $G_2$: by the definition (iii) of $<$, we easily see that the initial monomial is the first one and squarefree. 
\item On each binomial in $G_n$: by the definition (i) of $<$, we see that the initial monomial is the first one. 
Moreover, by definition of each binomial in $G_n$, since $\bigcup_{j=1}^n \widetilde{{\bf s}_j}\supset\widetilde{{\bf u}_0}$ 
(resp. $\bigcup_{j=1}^n \widetilde{{\bf s}_j'}\supset\widetilde{{\bf u}_m}$) 
but $\bigcup_{\substack{1 \leq j \leq n \\ j \neq i}}\widetilde{{\bf s}_j}\not\supset \widetilde{{\bf u}_0}$, 
(resp. $\bigcup_{\substack{1 \leq j \leq n \\ j \neq i}}\widetilde{{\bf s}_j'}\not\supset \widetilde{{\bf u}_m}$), 
we also see that the first monomial is squarefree. 
\end{itemize}

Our goal is to prove that $G$ forms a Gr\"{o}bner basis of $I$ with respect to the monomial order $<$ defined above. 
It is easy to see that $G \subset I$. It follows from \cite{OH02} that, in order to prove that $G$ is a Gr\"{o}bner basis of $I$, 
we may prove the following assertion: 
\begin{center}
If $u$ and $v$ are monomials belonging to $S$ with $u \neq v$ \\
such that $u \notin \ini_{<}(G)$ and $v \notin \ini_{<}(G)$, then $\pi(u) \neq \pi(v) $, 
\end{center}
where $\ini_<(G)=(\ini_<(g) \;|\; g \in G)$. 

Let $u,v\in S$ be monomials such that $u \notin \ini_{<}(G)$ and $v \notin \ini_{<}(G)$. 
Since each of $u$ and $v$ is not divisible by the initial monomials in $G$, those must be of the forms 
\[u=x_{{\bf s}_1}x_{{\bf s}_2}\cdots x_{{\bf s}_\ell}y_i^{c_1}{y_{i+1}}^{c_2},\; 
v=x_{{\bf s}_1'} x_{{\bf s}_2'} \cdots x_{{\bf s}_{\ell'}'}{y_{i'}}^{c_1'}{y_{{i'}+1}}^{{c_2}'},\]
where 
\begin{itemize}
\item $0 \leq i,i' \leq m-1$, $c_1,c_2,c_1',c_2' \geq 0$ 
(note that any monomial $y_iy_{i+\epsilon}$ with $\epsilon \geq 2$ can be divisible by the initial monomial in $G_2$); 
\item $x_{{\bf s}_1}x_{{\bf s}_2}\ldots x_{{\bf s}_\ell}$ and $x_{{\bf s}_1'}x_{{\bf s}_2'} \ldots x_{{\bf s}_{\ell'}'}$ are sorted 
(otherwise, it follows from the property of the sorting order that 
$x_{{\bf s}_1}x_{{\bf s}_2}\ldots x_{{\bf s}_\ell}$ or $x_{{\bf s}_1'}x_{{\bf s}_2'}\ldots x_{{\bf s}_{\ell'}'}$ 
is divisible by an initial monomial of $G_{1,1} \cup G_{1,2} \cup G_{1,3}$);
\item $\bigcup_{j=1}^\ell\widetilde{{\bf s}_j} \not\supset \widetilde{{\bf u}_0}$, 
$\bigcup_{j=1}^{\ell'}\widetilde{{\bf s}_j'} \not\supset \widetilde{{\bf u}_0}$ and 
$\bigcup_{j=1}^\ell\widetilde{{\bf s}_j} \not\supset \widetilde{{\bf u}_m}$, 
$\bigcup_{j=1}^{\ell'}\widetilde{{\bf s}_j'} \not\supset \widetilde{{\bf u}_m}$ 
(otherwise, we can see that $x_{{\bf s}_1} x_{{\bf s}_2}\cdots x_{{\bf s}_{\ell}}$ and $x_{{\bf s}_1'} x_{{\bf s}_2'}\cdots x_{{\bf s}_{\ell'}'}$ 
are divisible by the initial monomial in $G_{1,2},G_{1,3}$ or $G_n$ for some $3 \leq n \leq k$). 
\end{itemize}
Suppose $\pi(u)=\pi(v)$. In what follows, we will show that $u=v$. 

Let $\pi(x_{{\bf s}_i})=u_{\alpha_i}s$ and $\pi(x_{{\bf s}_i'})=u_{{\alpha_i}'}s$.
Since 
$$\pi(x_{{\bf s}_1}x_{{\bf s}_2}\cdots x_{{\bf s}_\ell})\pi(y_i^{c_1}{y_{i+1}}^{c_2})
=\pi(x_{{\bf s}_1'} x_{{\bf s}_2'} \cdots x_{{\bf s}_{\ell'}'})\pi(y_{i'}^{c_1'}y_{i'+1}^{c_2'})$$
and 
$\pi(y_j)=t_1 \cdots t_{k-1}t_k^j \cdots t_d^js$ for each $j$, we can write 
\begin{align}\label{siki}
(\alpha_1+\cdots+\alpha_\ell)-(\alpha_1'+\cdots+\alpha_{\ell'}')
=(\underbrace{a,\cdots,a}_{k-1},\underbrace{b,\cdots,b}_k) \in \ZZ^d, 
\end{align}
where $\alpha_j=v_{s_{j,1}}+\cdots+v_{s_{j,k}}$ and $\alpha_j'=v_{s_{j,1}'}+\cdots+v_{s_{j,k}'}$ for each $j$, 
$a=c_1'+c_2'-c_1-c_2$ and $b=c_1'i'+c_2'(i'+1)-(c_1i+c_2(i+1))$. 

For each $k \leq j \leq d$, let 
$$e_j=\sum_{i=1}^\ell |\{r \;|\; s_{i,r}=j, 1 \leq r \leq k\}|-\sum_{i=1}^{\ell'}|\{ r \;|\; s_{i,r}'=j, 1 \leq r \leq k\}|.$$ 
Then the $d$-th coordinate of the left-hand side of \eqref{siki} is equal to $me_d$ 
by the form of each of $v_0,\ldots,v_d$. Thus, $b=me_d$. 

Assume that $e_d$ is non-negative. Then $b \geq 0$. If there is $d$ in $\{k,k+1,\ldots,d\} \setminus \bigcup_{i=1}^{n}\widetilde{{\bf s}_i}$, 
then we can see that the $d$-th coordinate of the left-hand side of \eqref{siki} is non-positive, so we have $b=e_d=0$. 
Otherwise, take $j \in \{k,k+1,\ldots,d\} \setminus \bigcup_{i=1}^\ell \widetilde{{\bf s}_i}$ with $j \neq d$. 
Remark that $\{k,k+1,\ldots,d\} \setminus \bigcup_{i=1}^\ell \widetilde{{\bf s}_i} \neq \emptyset$ since 
$\bigcup_{i=1}^\ell \widetilde{{\bf s}_i} \not\supset \widetilde{{\bf u}_m}$. 
Then the $j$-th coordinate of the left-hand side of \eqref{siki} is equal to $(m-1)e_d+e_j$, 
where $e_j$ is non-positive by $j \not\in \bigcup_{i=1}^\ell \widetilde{{\bf s}_i}$. 
Thus $b=me_d=(m-1)e_d+e_j$ implies that $b=e_d=0$. 
Similarly, even if $e_d$ is non-positive, by repacing the roles of ${\bf s}_i$'s and ${\bf s}_i'$'s, we conclude that $b=e_d=0$. 

On the other hand, by comparing the degrees of $u$ and $v$, we can see that $\ell +c_1+c_2=\ell'+{c_1}'+{c_2}'$. 
Moreover, since ${c_1}'+{c_2}'-c_1-c_2=a$, we obtain $\ell - \ell'=a$. 
Assume that $a \geq 0$. (The case $a \leq 0$ is similar.) By \eqref{siki} and $b=e_d=0$, we see that 
\begin{align*}\mathrm{sort}({\bf s}_1 \cdots {\bf s}_\ell) =\mathrm{sort}({\bf s}_1'\cdots {\bf s}_{\ell'}'
\underbrace{0 \cdots 0}_a \underbrace{1 \cdots 1}_a\cdots\underbrace{k-1 \cdots k-1}_a).\end{align*}
From $\bigcup_{j=1}^{\ell}\widetilde{{\bf s}_j} \not\supset \widetilde{{\bf u}_0}$, we have $a=0$. 
Therefore, we see that $\ell=\ell'$. Since $v_1,\ldots,v_d$ are linearly independent 
and both ${\bf s}_1{\bf s}_2\ldots {\bf s}_\ell$ and ${\bf s}_1'{\bf s}_2'\ldots{\bf s}_{\ell'}'$ are sorted, 
we conclude that $x_{{\bf s}_1}x_{{\bf s}_2} \cdots x_{{\bf s}_\ell}=x_{{\bf s}_1'} x_{{\bf s}_2'} \cdots x_{{\bf s}_{\ell'}'}$. 

Our remaining task is to check that $y_i^{c_1}{y_{i+1}}^{c_2}={y_{i'}}^{{c_1}'}{y_{{i'}+1}}^{{c_2}'}$. 
From $a=b=0$, we know that 
$c_1'+c_2'=c_1+c_2$ and $c_1'i'+c_2'(i'+1)=c_1i+c_2(i+1)$. Without loss of generality, we may assume $i' \geq i$. 
By deleting $c_1'$, we obtain $(c_1+c_2)(i'-i)=c_2-c_2'$. 
\begin{itemize}
\item When $i'=i$, since $c_2-c_2'=0$, we obtain that $c_2=c_2'$, and thus $c_1=c_1'$. 
Hence, $y_i^{c_1}y_{i+1}^{c_2}=y_{i'}^{c_1'}y_{i'+1}^{c_2'}$. 
\item When $i'=i+1$, we have $c_1+c_2=c_2-c_2'$. Hence, we see that $c_1=-c_2'$, i.e., $c_1=c_2'=0$. 
Thus, $c_1'=c_2$. Therefore, we obtain $y_i^{c_1}y_{i+1}^{c_2}=y_{i'}^{c_1'}y_{i'+1}^{c_2'}=y_{i+1}^{c_2}.$ 
\item Assume $i' - i \geq 2$. Then we have $2(c_1+c_2) \leq c_2-c_2'$, i.e., $2c_1+c_2 +c_2' \leq 0$, 
i.e., $c_1=c_2=c_2'=0$ and $c_1'=0$. Hence, we obtain that $y_i^{c_1}y_{i+1}^{c_2}=y_{i'}^{c_1'}y_{i'+1}^{c_2'}=1$. 
\end{itemize}
Consequently, we conclude that $y_i^{c_1}{y_{i+1}}^{c_2}={y_{i'}}^{{c_1}'}{y_{{i'}+1}}^{{c_2}'}$, as required. 
\end{proof}

\bigskip


\end{document}